\documentclass[11pt]{article}

\usepackage[top = 3cm, left=2.5cm, right=2.5cm, bottom=3.3cm]{geometry}
\usepackage[english]{babel}
\usepackage[nodayofweek,level]{datetime}
\usepackage{enumitem}   
\usepackage{caption}
\usepackage{subcaption}
\usepackage{graphicx}
\usepackage{xcolor}

\usepackage{amsthm}
\usepackage{amssymb}
\usepackage{amsmath}

\usepackage{algorithmic}
\usepackage[ruled, linesnumbered]{algorithm2e}
\SetKwInput{KwInput}{Input}
\SetKwInput{KwOutput}{Output}

\newcommand\numberthis{\addtocounter{equation}{1}\tag{\theequation}}

\usepackage{fancyhdr}
\usepackage{dirtytalk}

\newtheorem{theorem}{Theorem}[section]

\newtheorem{lemma}[theorem]{Lemma}

\newtheorem{claim}[theorem]{Claim}

\usepackage{float}

\usepackage{hyperref}
\hypersetup{
	colorlinks,
	linkcolor={red!60!black},
	citecolor={green!50!black},
	urlcolor={blue!80!black}
}

\title{Solving Static Permutation Mastermind using \(O(n \log n)\) Queries}

\author{
	Maxime Larcher\thanks{Department of Computer Science, ETH Z\"urich, Switzerland \newline Email: \{\texttt{larcherm}\textbar\texttt{anders.martinsson}\textbar
     	\texttt{steger}\}\texttt{@inf.ethz.ch}}
  	\and
  	Anders Martinsson\footnotemark[1]
  	\and
  	Angelika Steger\footnotemark[1]
}

\begin{document}

\maketitle

\abstract{{Permutation Mastermind} is a version of the classical mastermind game in which the number of positions $n$ is equal to the number of colors $k$, and repetition of colors is not allowed, neither in the codeword nor in the queries.  In this paper we solve the main open question from Glazik, J\"ager, Schiemann and Srivastav (2021), who asked whether their bound of $O(n^{1.525})$ for the static version can be improved to \(O(n \log n)\), which would be best possible. By using a simple probabilistic argument we show that this is indeed the case.
}

\section{Introduction}
\label{sec:intro}

Mastermind is a well known two-player board game, whose commercial version was introduced in the 1970's by Mordecai Meirowitz. The game goes as follows: the first player, named \emph{Codemaker}, secretly chooses a codeword made up of \(4\) pegs, each of \(6\) possible colours. The second player, \emph{Codebreaker}, then tries to guess this codeword in as few queries as possible. A query also consists of \(4\) pegs of \(6\) possible colours, and for each query Codebreaker receives the number of pegs they guessed correctly (black pegs) as well as the number pegs of the correct colour, but not at the right position (white pegs).

In 1977, Knuth~\cite{knuth1976computer} used a minimax argument to prove that five queries were sufficient to find any codeword. Numerous variants of this game have been studied in the literature. Some of the most common variants are: 
\begin{enumerate}[label=(\roman*)]
	\item the number of pegs \(n\) and the number of colours \(k\) may be arbitrary;
	\item in the \emph{black-peg only} setting, Codebreaker only receives the black pegs as answer to their queries;
	\item when \(k \ge n\), one may or may not allow \emph{repetitions} of colours in the codeword and the queries; 
	\item in the \emph{static} setting, Codebreaker needs to decide all queries in advance, before receiving any answer. In the \emph{adaptive} setting, queries may be adapted depending on the answers to previous queries.
\end{enumerate} 

One of the first problem of this type to receive attention is \emph{Coin-Weighing}, a problem introduced in 1960 by Shapiro and Fine~\cite{shapiro1960e1399}, which can be proven equivalent to \(n\)-peg, \(2\)-colour Mastermind. In 1963, Erd\H{o}s and R\'enyi~\cite{erdos1963two} showed via a probabilistic argument that \(O \left(  n / \log n \right)\) queries were sufficient to crack the code. Building up on their idea, Chv\'atal~\cite{chvatal1983mastermind} showed the existence of a strategy for black-peg only Mastermind with repetition, using \(O(n \log k / \log n)\) queries when \(k \le n^{1-\varepsilon}\). This matches the information-theoretic lower bound in both the adaptive and static cases up to a constant factor. 

Determining the minimum number of queries for static, black-peg only Mastermind with repetitions is part of the wider problem of determining the \emph{metric dimension of graphs}. In a recent paper, Jiang and Polyanskii~\cite{jiang2019metric} determined, up to lower order terms, the \emph{metric dimension} of large powers of fixed graphs. Concerning Mastermind, their result imply that for any \emph{constant} number of colours \(k\), the query complexity is \((2 + o(1)) n / \log_k n\). 

For larger \(k\)'s, determining the correct query-complexity turned out to be more delicate. When \(k = n\), Chv\'atal's approach can easily be adapted to find a bound of \(O(n \log n)\), a factor $\log n$ away from the information-theoretic lower bound of $\Omega(n)$ (attributed to Duchet, see~\cite{doerr2016playing}). Doerr, Doerr, Sp\"ohel and Thomas~\cite{doerr2016playing} showed that this was the correct bound (up to constant factors) in the static case. They also presented an adaptive strategy using \(O(n \log \log n)\) queries, by giving a reduction to a sequence of Coin-Weighings problems. The gap to the lower bound was finally closed in a recent paper by Martinsson and Su~\cite{martinsson2020mastermind}, who showed that by first performing a sequence of \(O(n)\) simple ``pre-processing'' queries, Coin-Weighing schemes can be modified to match the information-theoretic lower bound for Mastermind. In the adaptive case, optimal bounds for \(k > n\) can be deduced from this bound, we refer to~\cite{doerr2016playing} and~\cite{martinsson2020mastermind} for the details.

In the case of Permutation Mastermind we assume\ \(k = n\) and allow no repetitions of colors, neither in the codeword nor in the queries. In this setting, one usually views the codeword and the queries as permutations of \([n]\), thus the name. It is trivial to see that black-peg only, and black- and white-peg settings are equivalent in this case. Restricting the codeword to permutations obviously makes it easier for Codebreaker to win, on the other hand this same restriction for queries rules out many of the usual strategies. In the adaptive setting, the best lower- and upper-bounds currently, of order \(\Omega(n)\) and \(O(n \log n)\), are both due to El Ouali, Glazik, Sauerland and Srivastav~\cite{el2018query}. Where the exact query-complexity lies inside this \(\Theta( \log n )\) gap is still an open question. 

Recently, Glazik, J\"ager, Schiemann and Srivastav~\cite{glazik2021bounds} shed some light on \emph{Static Permutation Mastermind}. 
Adapting an argument of Doerr, Doerr, Sp\"ohel and Thomas~\cite{doerr2016playing}, they proved that any strategy requires at least \(\Omega(n \log n)\) queries. Additionally, they show that there exists a strategy using \(O(n^{1.525})\) queries and ask whether it is possible to find a strategy matching their lower bound. In this paper we answer this question positively.

\begin{theorem}
	\label{thm:main_thm}
	For \emph{Static Permutation Mastermind}, there exists a set of \(O(n \log n)\) queries from which any codeword can be recovered.
\end{theorem}

Our proof uses the probabilistic method. We show that with positive probability, a random sequence of \(28 n \log n\) queries chosen uniformly at random has the property that it can uniquely determine all codewords, hence there are deterministic sequences with this property. 

Additionally, one easily observes that {\em given such a good sequence}, our proof of Theorem~\ref{thm:main_thm} can be turned into a simple algorithm which recovers the codeword \emph{efficiently} from such a sequence (by finding one colour at a time in \(n\) turns). This contrasts with Chv\'atal's original result for $k\leq n^{1-\varepsilon}$, where no corresponding efficient reconstruction scheme has been presented in the literature, leaving codebreaker to brute-force a codeword that matches all queries. In the case of static classical Mastermind with \(k = n\), our approach could be adapted have an efficient reconstruction algorithm, but we will not elaborate on this further.

For a quick overview of the state of the art on Mastermind, we provide the best known lower- and upper-bounds for the case \(k = n\) in Table~\ref{tab:summary}. The rest of this paper is dedicated to the proof of Theorem~\ref{thm:main_thm}. 

\begin{table}[H]
	\begin{center}
		\begin{tabular}{ l | c | c }
									& Lower bound 									& Upper bound\\
			\hline
			Adaptive Classical 		& \(\Omega(n)\) (see~\cite{doerr2016playing})	& \(O(n)\)~\cite{martinsson2020mastermind} \\ 
			Static Classical 		& \(\Omega(n \log n)\)~\cite{doerr2016playing} 	& \(O(n \log n)\)~\cite{chvatal1983mastermind} \\
			\hline
			Adaptive Permutation 	& \(\Omega(n)\)~\cite{el2018query}				& \(O(n \log n)\)~\cite{el2018query} \\
			Static Permutation 		& \(\Omega(n \log n)\)~\cite{glazik2021bounds}	& \(O(n \log n)\) (our paper) \\
		\end{tabular}
	\end{center}
	\caption{Summary of results for variations of Mastermind with $n$ colours and pegs. \label{tab:summary} }
\end{table}

\section{Proof of the Theorem}
\label{sec:proof}

We use standard notations. By \(\log\) we denote the natural logarithm. For the sake of conciseness we omit floor and ceil signs. Whenever it is needed, we assume that \(n\) is large enough. All permutations are assumed to be over \([n] = \{1, \dots, n\}\). 

Our proof of Theorem~\ref{thm:main_thm} is probabilistic and in certain ways resembles the original proof of Chv\'atal. Consider a set of queries \(Q\) of size $q(n)$ chosen independently and uniformly at random from all permissible queries. If we could show that it allows to decode a randomly chosen codeword with probability \(1 - o(1/n!) \), then a simple union bound argument allows us to conclude that there has to {\em exist} such a set of queries of size $q(n)$ that allows to decode {\em all} codewords.

Now, as it turns out, a success probability $1-o(1/n!)$ in this case is too optimistic. For instance, with probability $1/poly(n)$ a collection of $O(n \log n)$ random queries will never query colors $1$ or $2$ in positions $1$ or $2$, in which case, distinguishing permutations beginning with $12$ and $21$ is impossible. Instead, to get the union bound to follow through, we need a more refined picture of what ``bad events'' could cause the recovery to fail. This is made formal in Lemma~\ref{lem:existence_of_set}; first, we introduce some definitions. 

Given a set \(I \subseteq [n]\), we say that \(c \in [n]^I\) is a \emph{colouring} of \(I\). We say that it is a \emph{valid colouring} of \(I\) if \(c(i) \ne c(j)\) for all \(i \ne j\), in particular \(c\) may be viewed as the restriction of a permutation to \(I\). Given a permutation \(\sigma\) and a colouring \(c\), we say that \(\sigma\) is a \emph{\(0\)-query} for \(c\) if \(\sigma(i) \ne c(i)\) for all \(i \in I\). Given a permutation  \(\sigma\) and two colourings \(c_1, c_2\), we say that \emph{\(\sigma\) discriminates \(c_1\) from \(c_2\)} on \(I\) if \(\sigma\) is a \(0\)-query for \(c_1\) but not for \(c_2\). The following lemma is the core of our argument.

\begin{lemma}
	\label{lem:existence_of_set}
	There exists a set \(Q\) of \(28 n \log n\) permutations such that the following holds. For any \(I \subseteq [n]\), any valid colouring \(v\) and any colouring \(c\) on \(I\) such that \(v(i) \ne c(i)\) for all \(i \in I\), there exists some \(q \in Q\) which discriminates \(v\) from \(c\).
\end{lemma}

We emphasise that in the above lemma \(c\) can be \emph{any} colouring while \(v\) needs be a valid one. We defer the proof of this lemma to the end of this section and now prove Theorem~\ref{thm:main_thm}.

\begin{proof}[Proof of Theorem~\ref{thm:main_thm}]
	\label{prf:main_thm_proof}
	Let \(Q\) be the set as in Lemma~\ref{lem:existence_of_set}. Our strategy to recover any codeword \(c_0\) is to find the colours one by one until we have found them all. We let \(I\) be the set of pegs for which we have \emph{not} found the correct colour; at the beginning \(I = [n]\).
	
	At any step when \(|I| \ge 1\), because we know the correct colours outside of \(I\), we can compute for each query \(q \in Q\) how many positions \(c_0\) and \(q\) colour identically in \(I\). Let \(Q_I = \{q_1, \dots, q_l\} \subseteq Q\) be the set of \(0\)-queries for \(c_0\) on \(I\). We claim that there must exist some \(i \in I\) such that \(\{c_0(i), q_1(i), \dots, q_l(i)\} = [n]\). Indeed, if this were not the case, then there would exist some colouring \(c\) (which need not be valid) such that \(c(i) \notin \{c_0(i), q_1(i), \dots, q_l(i)\}\) for all $i\in I$.  However, this would imply that no \(q \in Q_I\) (and hence, in \(Q\)) discriminates \(c_0\) (or rather its restriction to \(I\)) from \(c\) on \(I\), and this contradicts the definition of \(Q\). In consequence, such an \(i \in I\) exists, and we can recover \(c_0(i)\): it is the unique element of \([n]\setminus \{q_1(i), \dots, q_l(i)\}\). We remove this \(i\) from \(I\) and continue until \(I = \emptyset\).
\end{proof}

All that remains is now to prove Lemma~\ref{lem:existence_of_set}.

\begin{proof}[Proof of Lemma~\ref{lem:existence_of_set} ]
	The proof is probabilistic: we choose independently and uniformly at random a set of \(28 n \log n\) queries and show that with positive probability, this set satisfies the conditions.
	
	To do this, we wish to apply union-bound over all \((I, v, c)\). A rough application will not work: indeed there are roughly \(n^{(2 + o(1))n}\) choices of triple \((I,v,c)\), but the probability of failing to discriminate a certain triple may be as large as $1/poly(n)$ if $I$ is small. To solve this problem, we partition the set of triples \((I,v,c)\) depending on the size of \(I\): when \(k = |I|\) is small, there are fewer choices for \((I, v, c)\) so it is fine if the probability of discriminating is smaller. 
	
	Concretely, assume that for any fixed triple \((I, v, c)\) as in the statement, a random colouring has probability at least \(|I|/7n\) of discriminating \(v\) from \(c\). Then the probability that none of the \(28 n \log n\) random queries discriminate is at most 
		\begin{align*}
			(1 - |I| / 7n)^{28 n \log n} \le e^{-4 |I| \log n} = n^{-4|I|}.
		\end{align*}
	When \(|I| = k\), a crude upper bound for the choice of \((I,v,c)\) is \(n^{3k}\) in total. Hence by union bound, the probability that there exists a triple \((I, v, c)\) with \(I \ne \emptyset\) for which no colouring is discriminating is at most
		\begin{align*}
			\sum_{k = 1}^{n}{ \sum_{|I| = k}{ \sum_{v, c}{n^{-4 |I|} } } } \le \sum_{k = 1}^{n}{n^{3k} n^{-4 k} } < 1,
		\end{align*}
	whenever \(n \ge 2\).
	
	Therefore, all we need to do is prove that, indeed, for any fixed \((I, v, c)\) as above, a uniformly random permutation has probability at least \(|I| /7 n\) of discriminating. To bound this, we denote by \(S_i\) the set of those permutations \(\sigma\) which are \(0\)-queries for \(v\) and such that \(\sigma(i) = c(i)\); we now want to show that the size of \(\bigcup_{i \in I}{S_i}\) is at least \(|I|(n-1)! / 7\).
	Using inclusion-exclusion, we get 
	\begin{align*}
		\left| \bigcup_{i \in I}{ S_i } \right| \ge \sum_{i \in I}{ |S_i| } - \sum_{i \ne j \in I}{ |S_i \cap S_j| }. \numberthis \label{eq:in_ex_expression_usi}
	\end{align*} 
	
	To estimate this quantity we shall compute the sizes of \(S_i\) and \(S_i \cap S_j\). We express them in terms of \(A(n, |I|)\) which we define as the number of \(0\)-queries of \(v\) on \(I\). If we let \(m = |I|\), write \(\mathfrak{S}_n\) for the set of all permutations on \([n]\) and \(T_i\) for the set of permutations \(\sigma\) such that \(\sigma(i) = v(i)\), then a simple application of the inclusion-exclusion principle gives
	\begin{align*}
		A(n, m) = \left| \mathfrak{S}_n \setminus \bigcup_{i \in I}{T_i} \right| = \sum_{k = 0}^{m}{ \left( (-1)^k {m \choose k} (n-k)! \right) }.
	\end{align*}

	\begin{claim}
		\label{clm:A}
		For all \(m \le n\) we have 
		\begin{align*}
			n! / 3 \le A(n, m) \le n!.
		\end{align*}
	\end{claim}
	
	\begin{proof}
		Since \(A(n,m)\) is the size of a subset of permutations of \([n]\), the upper bound
		is trivial. For the lower bound recall that the expression of \(A(n,m)\) stems from an inclusion-exclusion argument. As it is well known, that inclusion-exclusion alternatively over- resp.\ underestimates we get
		\begin{align*}
			A(n,m) \ge \sum_{k = 0}^{3}{ (-1)^k { m \choose k} (n-k)! }.
		\end{align*}
		By spelling out these terms and grouping them appropriately we get $A(n,m) \ge (1-\frac{m}n)n! + \frac{m(m-1)}{6n(n-1)} \left( 3-\frac{m-2}{n-2} \right)n!$, from which the claimed bound follows easily by observing that the last bracket is at least two and the remaining terms are monotonous in $m$.	
	\end{proof}
	
	Using the above, we now find expressions for the sizes of \(|S_i|, |S_i \cap S_j|\).
	\begin{claim}
		\label{clm:sizes}
		For all \(i \ne j \in I\), we have 
		\begin{enumerate}[label=(\roman*)]
			\item \(|S_i| \ge (1 - o(1)) A( n, |I| ) / n\); \label{clm_itm:si_size}
			\item \(|S_i \cap S_j| \le (1 + o(1)) A( n, |I| ) / n(n-1) \). \label{clm_itm:sisj_size}
		\end{enumerate}
	\end{claim}

	\begin{proof}
		We prove this claim by appropriate double counting arguments. For a 
		permutation $\sigma$ that is a $0$-query for $v$ we denote by $\ell$ the index such that $\sigma(\ell)=c(i)$. If $\ell\not\in I$ or  $\sigma(i) \not= v(\ell)$ then the permutation that is obtained from $\sigma$ by swapping the colours at positions $i$ and $\ell$ belongs to $S_i$. (Note that this case includes the case $i=\ell$, as then $\sigma$ belongs to $S_i$ already.) As there exist only $n(n-2)!$ permutations for which $\sigma(i) = v(\ell)$, we thus have 
		\begin{align*}
			A(n, |I|) - o\left(n!\right) \le n |S_i|.
		\end{align*} 
		
		For \ref{clm_itm:sisj_size}, we first note that \(S_i \cap S_j = \emptyset\) when \(c(i) = c(j)\), thus in this case the claim holds trivially. In the following we thus assume \(c(i) \ne c(j)\). Consider an arbitrary  \(\sigma \in S_i \cap S_j\). For some (arbitrary) indices $\ell_1$ and $\ell_2$,  perform the following operation: swap the colours at position $i$ and $\ell_1$ and at positions $j$ and $\ell_2$. When do we get a $0$-query for $v$? Clearly, a sufficient condition is to ensure that $\ell_1$ and $\ell_2$ are different from $i$ and $j$, the colours of $\sigma(\ell_1)$ and $\sigma(\ell_1)$ do not coincide with $v(i)$ resp.\ $v(j)$, and -- for $\ell_1, \ell_2\in I$  -- the colors of $\sigma(i)$ and $\sigma(j)$ does not coincide with that of
		$v(\ell_1)$ resp.\ $v(\ell_2)$. Clearly,  we thus have at least $(n-4)(n-5)$ proper choices for $\ell_1$ and $\ell_2$. Observe in addition that for each permutation $\sigma'$ that is obtained by such a double swap of $i,\ell_1$ and $j,\ell_2$, there is a unique \((\sigma, \ell_1, \ell_2)\) from which it can be obtained: \(\ell_1\) (resp.\ \(\ell_2\)) is the index such that \(\sigma'(\ell_1) = c(i)\) (resp.\ \(\sigma'(\ell_2) = c(j)\)) and \(\sigma\) is obtained from \(\sigma'\) by swapping the colours of \(i,j\) with those of \(\ell_1, \ell_2\) respectively. Thus we have
		\begin{align*}
			(n-4)(n-5) |S_i \cap S_j| \le A(n, |I|).
		\end{align*}
				
		The formulas as in \ref{clm_itm:si_size} and \ref{clm_itm:sisj_size} follow immediately from the above and the bounds of Claim~\ref{clm:A}, where we catch the change to $n(n-1)$ in the $o(\cdot)$-notation. 
	\end{proof}
		
	Putting the expressions of \(|S_i|, |S_i \cap S_j|\) from Claim~\ref{clm:sizes} into \eqref{eq:in_ex_expression_usi} gives 
	\begin{align*}
		\left| \bigcup_{i \in I}{ S_i } \right|
			&\ge \sum_{i \in I}{ |S_i| } - \sum_{i \ne j \in I}{ |S_i \cap S_j| } \\
			&\ge |I| \frac{A(n,|I|)}{n} - {I \choose 2} \frac{A(n, |I|)}{n(n-1)} + o\left(|I|(n-1)! \right) \\
			&= \frac{|I|}{n} A( n, |I| ) \left( 1 - \frac{|I|-1}{2(n-1)} \right) + o\left( |I|(n-1)! \right) \\ 
			&\ge ( 1/6 + o(1) ) |I| (n-1)! & \text{by Claim~\ref{clm:A}.}
	\end{align*}
	For a choice of \(n\) large enough, this is at least \(|I|(n-1)!/7\) as claimed.
\end{proof}

\bibliographystyle{abbrv}
\bibliography{mastermind}

\end{document}